\documentclass[a4paper,reqno,oneside]{amsart}

\usepackage{fullpage}
\usepackage{amsthm,amssymb,amsmath}
\usepackage{enumerate,enumitem}
\usepackage[font=small]{caption}
\usepackage{mathtools}
\usepackage{graphicx}
\usepackage{xcolor}

\makeatletter
\@namedef{subjclassname@2020}{\textup{2020} Mathematics Subject Classification}
\makeatother

\newtheorem{theorem}{Theorem}[section]
\newtheorem{corollary}[theorem]{Corollary}
\newtheorem{lemma}[theorem]{Lemma}
\newtheorem{proposition}[theorem]{Proposition}

\theoremstyle{definition}

\theoremstyle{remark}
\newtheorem{remark}[theorem]{Remark}
\newtheorem*{remark*}{Remark}

\numberwithin{equation}{section}
\renewcommand{\theequation}{\arabic{section}.\arabic{equation}}

\newcommand{\Uu}{{\mathcal U}}

\newcommand{\N}{\mathbb{N}}
\newcommand{\Z}{\mathbb{Z}}

\newcommand{\R}{\mathbb{R}}
\newcommand{\C}{\mathbb{C}}

\newcommand{\la}{\lambda}
\newcommand{\eps}{\varepsilon}
\newcommand{\ii}{{\rm i}}

\newcommand{\dd}{\,{\rm d}}

\newcommand{\re}{{\rm Re }}
\newcommand{\im}{{\rm Im }}

\mathtoolsset{showonlyrefs}

\usepackage{color}
\usepackage[normalem]{ulem}
\definecolor{DarkBlue}{rgb}{0,0.1,0.7}  
\definecolor{DarkGreen}{rgb}{0,0.5,0.1}

\newcommand\soutD{\bgroup\markoverwith
	{\textcolor{DarkGreen}{\rule[.5ex]{2pt}{1pt}}}\ULon}
\newcommand{\Hm}[1]{\leavevmode{\marginpar{\tiny%
			$\hbox to 0mm{\hspace*{-0.5mm}$\leftarrow$\hss}%
			\vcenter{\vrule depth 0.1mm height 0.1mm 
		    width \the\marginparwidth}%
			\hbox to
			-0.5mm{\hss$\rightarrow$\hspace*{-0.5mm}}$\\
			\relax\raggedright #1}}}

\title{Criticality transition for positive powers of the discrete Laplacian on the half line}

\author{Borbala Gerhat}
\address[Borbala Gerhat]{Department of Mathematics, Faculty of Nuclear Sciences and Physical Engineering, Czech Technical University in Prague, Trojanova 13, 120 00 Praha, Czech Republic}
\email{borbala.gerhat@fjfi.cvut.cz}

\author{David Krej\v ci\v r\'ik}
\address[David Krej\v ci\v r\'ik]{Department of Mathematics, Faculty of Nuclear Sciences and Physical Engineering, Czech Technical University in Prague, Trojanova 13, 120 00 Praha, Czech Republic}
\email{david.krejcirik@fjfi.cvut.cz}

\author{Franti\v sek \v Stampach}
\address[Franti\v sek \v Stampach]{Department of Mathematics, Faculty of Nuclear Sciences and Physical Engineering, Czech Technical University in Prague, Trojanova 13, 120 00 Praha, Czech Republic}
\email{frantisek.stampach@fjfi.cvut.cz}



\date{July 18, 2023}

\begin{document}

\begin{abstract}
We study the criticality and subcriticality of powers $(-\Delta)^\alpha$ with $\alpha>0$ of the discrete Laplacian $-\Delta$ acting on $\ell^2(\N)$. We prove that these positive powers of the Laplacian are critical if and only if $\alpha \ge 3/2$. We complement our analysis with Hardy type inequalities for $(-\Delta)^\alpha$ in the subcritical regimes $\alpha \in (0,3/2)$. As an illustration of the critical case, we describe the negative eigenvalues emerging by coupling the discrete bilaplacian with an arbitrarily small localized potential.
\end{abstract}
	
\maketitle

\section{Introduction}

\subsection{Physical motivation}
The uniqueness of the world we live in consists in that $\R^3$
is the lowest dimensional Euclidean space for which
the Brownian motion is \emph{transient}.
Indeed, it is well known that the Brownian particle in $\R^d$
will escape from any bounded set after some time forever if $d \geq 3$,
while the opposite holds true in low dimensions,
i.e.~the Brownian motion is \emph{recurrent} in~$\R^1$ and~$\R^2$.
This is a well known criticality transition in dimensions.

Since the Brownian motion is mathematically introduced
via the heat equation, 
it is not surprising that the transiency is closely related to
spectral properties of the Laplacian.
Indeed, the self-adjoint realization $-\Delta$ in $L^2(\R^d)$
is \emph{subcritical} if and only if $d \geq 3$,
meaning that there exists a non-trivial non-negative function~$V$
such that the Hardy-type inequality $-\Delta \geq V$ holds
in the sense of quadratic forms. 
On the other hand, $-\Delta$ is critical if $d=1,2$
in the sense that $\inf \sigma(-\Delta+V) < 0$ 
for every non-trivial non-positive function~$V$. 
The Hardy inequality has other physical consequences,
namely in quantum mechanics where
it can be interpreted in terms of the uncertainty principle
and leads to the stability of matter in~$\R^3$.

The case of Brownian particles dying on massive subsets of~$\R^d$ is less interesting
in the sense that the Dirichlet Laplacian $-\Delta$
in $\R^d \setminus \overline\Omega$
with any~$\Omega$ non-empty and open is always subcritical.  
In particular, the Brownian motion in the half-space 
$\R^{d-1} \times (0,\infty)$ is transient for every $d \geq 1$,
so no criticality transition in dimensions occurs. 

There is a probabilistic interpretation of powers of the Laplacian
in terms of an anomalous diffusion. From this perspective,
the case of the half-line is equally uninteresting because 
all the powers $(-\Delta)^k$ in $L^2(0,\infty)$
with $k \in \N$ are subcritical.  
There is no criticality transition in powers.

The objective of this paper is to disclose the surprising fact
that the situation is very different in the discrete setting. 
Indeed, we demonstrate that the integer powers of 
the discrete Laplacian $(-\Delta)^k$ on $\ell^2(\N)$
are subcritical if and only if $k =1$.
What is more curious in fine properties of this transition, 
we consider possibly non-integer powers 
and reveal the following precise threshold in all positive powers:
\begin{center}
$(-\Delta)^\alpha$ on $\ell^2(\N)$ is subcritical if and only if
$\alpha < 3/2$.
\end{center}

\subsection{Mathematical formulation}

The discrete Laplacian on the (discrete) half line $\N=\{1,2,3,\dots\}$ is defined as the second-order difference operator $-\Delta$ given by the formula
\begin{equation}
	(-\Delta u)_n := - u_{n-1} + 2 u_n - u_{n+1}, \quad n\in\N,
\end{equation}
where $u=\{u_{n}\}_{n=1}^{\infty}$ is a complex sequence, together with the convention $u_{0}:=0$. When regarded as an operator on the Hilbert space $\ell^2(\N)$, the discrete Laplacian is bounded and self-adjoint with spectrum $\sigma (-\Delta) = [0,4]$. The matrix representation of $-\Delta$ with respect to the standard basis of $\ell^2(\N)$ coincides with the tridiagonal Toeplitz matrix
\begin{equation}\label{eq:def.Lap}
	-\Delta=\begin{pmatrix}
		2 & -1 \\
		-1 & 2 & -1\\
		& -1 & 2 & -1\\
		& & \ddots & \ddots & \ddots
	\end{pmatrix}.
\end{equation}

It is well known that $-\Delta$ is \emph{subcritical} 
meaning that there exists a non-trivial
diagonal operator $V\geq0$ such that $-\Delta \geq V$
in the sense of quadratic forms. 
Indeed, one has the classical Hardy inequality
$
 -\Delta\geq V^{\text{H}},
$
where
\begin{equation}
 V^{\text{H}}_{n}:=\frac{1}{4n^{2}}, \quad n\in\N.
\label{eq:classical_hardy_weight}
\end{equation}
Interestingly, even though the constant $1/4$ 
in the Hardy weight is optimal, the shifted operator
$-\Delta-V^{\text{H}}$ is still subcritical. An improved Hardy-type inequality 
$-\Delta\geq V^{\text{KPP}}$ with
$$
V^{\text{KPP}}_{n}:=2-\sqrt{\frac{n-1}{n}}-\sqrt{\frac{n+1}{n}}, \quad n\in\N,
$$
was found only recently by Keller, Pinchover, and Pogorzelski in~\cite{Keller-Pinchover-Pogorzelski-2018}; see also~\cite{Krejcirik-Stampach-2022} for a simple proof.
Moreover, it is proved in these references that 
$-\Delta-V^{\text{KPP}}$ is \emph{critical} in the sense
of the spectral instability
$\inf \, \sigma(-\Delta-V^{\text{KPP}}+V) < 0$ 
for any non-trivial diagonal $V \leq 0$.
In fact, not only the criticality, but an optimality of the weight $V^{\text{KPP}}$ in a stronger sense was proven in~\cite{Keller-Pinchover-Pogorzelski-2018} together with more general results on discrete Laplacians on graphs; see also~\cite{Keller-Pinchover-Pogorzelski-2018-AMM,Keller-Pinchover-Pogorzelski-2020} and moreover~\cite{Devyver-Fraas-Pinchover-2014, Devyver-Pinchover-2016} for related works in the continuous framework.
We emphasize the contrast with the (continuous) Dirichlet Laplacian on the half line, where $-\Delta-V^{\text{H}}$ with the classical Hardy weight $V^{\text{H}}(x):=1/(4x^{2})$ is critical in $L^{2}(0,\infty)$.

While the (sub)criticality and the related Hardy-type inequalities for $-\Delta$ on $\ell^{2}(\N)$ are well understood, much less is known about their generalization to the discrete polyharmonic operator or any positive power of $-\Delta$. The  primary goal of this paper is to investigate the criticality or subcriticality of $(-\Delta)^\alpha$ depending on $\alpha>0$, and to make the first attempt towards
Hardy-type inequalities in the subcritical regimes. 
Except for a partial result in~\cite{Huang-Ye-2022-arxiv} for $\alpha =2$, the topic seems not to be studied for any non-trivial exponent $\alpha \neq 1$ so far.

The paper is organized as follows. In Subsection~\ref{subsec:main}, our main results are formulated as Theorems~\ref{thm:main1}, \ref{thm:main2}, and \ref{thm:main3}. Subsection~\ref{subsec:rel_liter} summarizes some relevant results on positive powers of Laplacians on the half line or the line in both the discrete and continuous settings. 
After Section~\ref{sec:prelim}, where preliminary results on powers of the Laplacian and their Green kernel are presented, the three main theorems are proven in Section~\ref{sec:proofs}. The particular case of the bilaplacian is further studied in Section~\ref{sec:bilap}. The paper is concluded by an appendix with proofs of auxiliary integral identities.

\subsection{Main results}\label{subsec:main}
In this paper, we adopt the following definition 
of subcriticality/criticality in terms of 
spectral stability/instability against small perturbations. 
Given any bounded self-adjoint operator $H$ on $\ell^{2}(\N)$,
we say that~$H$ is \emph{critical} if $\inf\sigma(H+V) < 0$
whenever $V \leq 0$ is non-trivial.
Generically throughout this paper, $V$ denotes a potential (or a weight), i.e.~a diagonal operator acting on $\ell^{2}(\N)$. As usual, our notation does not distinguish between a diagonal operator $V$ on $\ell^{2}(\N)$ and the sequence $V=\{V_{n}\}_{n=1}^{\infty}$ of its diagonal entries.
We say~$H$ is \emph{subcritical} if it is not critical,
which is equivalent to the existence of a non-trivial weight $V \geq 0$ 
such that $H \geq V$ in the sense of forms
(i.e.\ $\langle \psi,(H-V)\psi\rangle\geq0$ 
for all $\psi\in\ell^{2}(\N)$).

The question of criticality of $(-\Delta)^{\alpha}$ for $\alpha>0$ is answered by our first main result.

\begin{theorem}\label{thm:main1}
 Suppose $\alpha>0$. Then $(-\Delta)^{\alpha}$ is critical if and only if $\alpha\geq3/2$.
\end{theorem}

The proof of Theorem~\ref{thm:main1} is given in Subsection~\ref{sec:proof1}.

\begin{remark}
Notice that $(-\Delta)^{\alpha}$ is bounded from above by $4^{\alpha}$ since $\sigma((-\Delta)^\alpha)=[0,4^{\alpha}]$. Therefore, one could also study the stability of the upper bound $4^{\alpha}$ when adding a perturbation $V\geq0$, i.e.~the criticality of the operator $4^{\alpha}-(-\Delta)^{\alpha}$
in our setting. 
We briefly comment on this in Subsection~\ref{subsec:crit_from_above}, where we prove $4^{\alpha}-(-\Delta)^{\alpha}$ to be subcritical for all $\alpha>0$. This result has no continuous analogue since the Dirichlet Laplacian on $L^{2}(0,\infty)$ is not bounded from above.
\end{remark}

For the particular case of the discrete bilaplacian $\Delta^{2}$ on $\ell^{2}(\N)$ it has already been observed in~\cite{Huang-Ye-2022-arxiv} that no direct analogue of the classical Rellich inequality holds, i.e.~that there exists \textbf{no} $c>0$ such that $\Delta^{2}\geq V^{\text{R}}$ with
$$
 V^{\text{R}}_{n}:=\frac{c}{n^{4}}, \quad n\in\N;
$$
this answers a question from~\cite{Gerhat-Krejcirik-Stampach-2023}.
Our Theorem~\ref{thm:main1} shows that there exists no discrete Rellich inequality on $\ell^{2}(\N)$, neither any Hardy-type inequality for the discrete polyharmonic operator $(-\Delta)^{k}$ on $\ell^{2}(\N)$ with $k=3,4\dots$.
This can be surprising when compared to the continuous setting, where the polyharmonic operator $(-\Delta)^{k}$ is subcritical in $L^{2}(0,\infty)$ for all $k \in \N$; see Subsection~\ref{subsec:rel_liter} below for more details.

Theorem~\ref{thm:main1} implies the existence of Hardy-type inequalities for $(-\Delta)^{\alpha}$ with $\alpha\in(0,3/2)$. A sufficient condition for admissible Hardy weights is given in the next theorem. For its formulation, we define the sequence of functions 
\begin{equation}\label{eq:def_g_n}
	g_n (\alpha) := \left(1-\frac{(\alpha)_{2n}}{(1-\alpha)_{2n}}\right)\tan(\pi\alpha)
\end{equation}
for $\alpha\in(0,3/2)$ and $n\in\N$, where $(\alpha)_{k}:=\alpha(\alpha+1)\dots(\alpha+k-1)$ is the Pochhhammer symbol. For $\alpha=1/2$ and $\alpha=1$, the values of $g_{n}(\alpha)$ are given by the respective limits
\begin{equation}\label{eq:g_n_1/2_1}
	g_n\left(\frac12\right) = \frac{2}{\pi} \sum_{j=1}^{2n} \frac{1}{2j-1} \quad\mbox{ and }\quad g_n\left(1\right) = 2\pi n.
\end{equation}
Notice that $g_{n}(\alpha)>0$ for all $\alpha\in(0,3/2)$ and $n\in\N$. In the sequel, $\Gamma$ denotes the Gamma function.

\begin{theorem}\label{thm:main2}
Let $\alpha\in(0,3/2)$. If a potential $V\geq0$ satisfies the condition
\begin{equation}
 \sum_{n=1}^{\infty}g_{n}(\alpha)V_{n}\leq 2\pi\frac{\Gamma(2\alpha)}{\Gamma^{2}(\alpha)}
\label{eq:V_suff_cond}
\end{equation}
where $g_{n}$ is as in~\eqref{eq:def_g_n} and~\eqref{eq:g_n_1/2_1}, then 
\[
 (-\Delta)^{\alpha}\geq V.
\]
\end{theorem}

The proof of Theorem~\ref{thm:main2} is worked out in Subsection~\ref{sec:proof2}.

In order to obtain more concrete Hardy-type inequalities in the subcritical regimes $\alpha\in(0,3/2)$, we apply Theorem~\ref{thm:main2} to deduce inequalities with Hardy weights which are suitable power functions. While our definition of subcriticality merely involves the existence of non-negative weights, we emphasize that the Hardy weights obtained in the theorem below are in fact strictly positive.

\begin{theorem}\label{thm:main3}
For every $\alpha\in(0,3/2)$ and $\varepsilon>0$, there exists a positive constant $\gamma=\gamma(\alpha,\varepsilon)$ such that 
\[
 (-\Delta)^{\alpha}\geq V,
\]
where
\begin{equation}
 V_{n}=V_{n}(\alpha,\varepsilon):=\frac{\gamma}{n^{\max(1,2\alpha)+\varepsilon}}.
\label{eq:power-form_Hardy_weight}
\end{equation}
(See~\eqref{eq:gamma_alpha_eps} below for an explicit formula for the constant $\gamma$.)
\end{theorem}

Theorem~\ref{thm:main3} is proven in Subsection~\ref{sec:proof3}.

\begin{remark}
A comparison of~\eqref{eq:power-form_Hardy_weight} with the classical Hardy weight~\eqref{eq:classical_hardy_weight} for $\alpha=1$ shows that $V_{n}$ does not have the optimal decay rate due to the presence of the positive parameter $\varepsilon$. If one moreover compares with the decay rate of the continuous weight~\eqref{eq:V_BD} below, one may believe that, at least for $\alpha\in(1/2,3/2)$, the inequality of Theorem~\ref{thm:main3} remains true when $\varepsilon=0$ with some positive constant $\gamma=\gamma(\alpha)$. The necessity of $\varepsilon>0$ is implied by our approach. In the course of our proof, we use a uniform bound on the Green kernel of $(-\Delta)^{\alpha}$ which is likely not sharp (see Lemma~\ref{lem:unif.bound.refined}). As a result, we get constants $\gamma(\alpha,\varepsilon)$ which explode when $\varepsilon\to0^{+}$, as one can see from the explicit formula~\eqref{eq:gamma_alpha_eps}. We therefore do not try to optimize the constant $\gamma(\alpha,\varepsilon)$. An open problem, which currently seems to be out of reach, is whether one can find $V=V(\alpha)\geq0$ such that $(-\Delta)^{\alpha}-V$ is critical for every value $\alpha\in(0,3/2)$.
\end{remark}

\subsection{Relevant literature}\label{subsec:rel_liter}

We briefly discuss several closely related results and summarize the state of the art concerning mainly the criticality of positive powers of Laplacians on the half or the full line in both the discrete and the continuous settings.

\medskip
1) \textbf{Discrete polyharmonic operators on $H_0^k(\N)$.} 
Let $k \in \N$ and let $\{e_{n} \mid n\in\N\}$ denote the standard basis of $\ell^2(\N)$. When the discrete polyharmonic operator $(-\Delta)^{k}$ is restricted to the subspace $H_0^k(\N):=\{e_{1},\dots,e_{k-1}\}^{\perp}$ of $\ell^{2}(\N)$, there exist discrete analogues of the continuous Birman inequality~\cite{Birman-1961,Glazman-1965} for the polyharmonic operator $(-\Delta)^k$ in $L^2(0,\infty)$, i.e.~the inequality $(-\Delta)^k \ge V^{{\rm B},k}$ with the weight
	\begin{equation}\label{eq:poly.weight}
		V^{{\rm B},k} (x) := \frac{((2k)!)^2}{16^k(k!)^2} \frac{1}{x^{2k}},
	\end{equation}
see also~\cite{Gesztesy-Littlejohn-Michael-Wellman-2018} for a recent proof. The discrete version of the Birman inequality $(-\Delta)^{k}\geq V^{\rm B,k}$ on $H_0^k(\N)$, where $V^{\rm B, k}_{n}$ is as in~\eqref{eq:poly.weight} with $x$ replaced by $n$, was proven in~\cite{Huang-Ye-2022-arxiv} (while it was deduced with a smaller constant in the PhD thesis ~\cite{Gupta-2023-PhD}, see also~\cite{Gupta-2021-arxiv}).
An improved discrete Birman inequality on $H_0^k(\N)$ with a weight strictly larger than $V^{\rm B,k}$ has been only conjectured in~\cite{Gerhat-Krejcirik-Stampach-2023} for $k\geq3$.

\medskip
2) \textbf{The discrete bilaplacian on $H_0^2(\N)$.} 
More is known about improved inequalities for the discrete bilaplacian $\Delta^{2}$ on $H^2_0(\N)$. In~\cite{Gerhat-Krejcirik-Stampach-2023}, the discrete Rellich inequality $\Delta^2 \ge V^{\rm GKS}$ was derived on $H^2_0(\N)$ with
	\begin{equation}
		V^{\rm GKS}_n := 6 - 4 \left(1 + \frac1n\right)^\frac32 - 4 \left(1 - \frac1n\right)^\frac32 + \left(1 + \frac2n\right)^\frac32 + \left(1 - \frac2n\right)^\frac32, \quad n \in \N.
	\end{equation}
This weight is asymptotically equal but strictly larger than the discrete analogue of $V^{\rm B, 2}$.
Further improvements upon $V^{\rm GKS}$ were obtained only recently in~\cite{Huang-Ye-2022-arxiv}. Nevertheless, a critical (or even optimal) discrete Rellich weight on $H_{0}^{2}(\N)$ remains unknown at the moment.

\medskip
3) \textbf{Discrete fractional Laplacians on $\ell^{2}(\Z)$.} When the discrete Laplacian $-\Delta$ is considered on the full line~$\Z$, the picture is more complete. For fractional powers $\alpha \in (0,1/2)$, it was shown in~\cite{Ciaurri-Roncal-2018} that $(-\Delta)^\alpha \ge V^{{\rm CR},\alpha}$ on $\ell^2(\Z)$ with the weight
		\begin{equation}\label{eq:V_CR}
			V^{{\rm CR},\alpha}_n := 4^\alpha \frac{\Gamma^2\left(\frac{1+2\alpha}{4}\right)}{\Gamma^2\left(\frac{1-2\alpha}{4}\right)}\frac{\Gamma \left(|n| + \frac{1-2 \alpha}{4}\right) \Gamma \left(|n| + \frac{3-2 \alpha}{4}\right)}{\Gamma \left(|n| + \frac{1+2 \alpha}{4}\right) \Gamma \left(|n| + \frac{3+2 \alpha}{4}\right)}, \quad n \in \Z.
		\end{equation}
Interestingly, the weight $V^{{\rm CR},\alpha}$ turns out to be optimal which was proven only recently in~\cite{Keller-Nietschmann-2023}. It particularly follows that $(-\Delta)^{\alpha}-V^{\rm CR, \alpha}$ is critical for all $\alpha\in(0,1/2)$. Although this seems not to be mentioned in the papers \cite{Ciaurri-Roncal-2018, Keller-Nietschmann-2023}, we remark without a proof that, for $\alpha\geq1/2$, the operator $(-\Delta)^\alpha$ is critical. This can be verified by 
the same method which we use in the proof of Theorem~\ref{thm:main1}. It means that there are no Hardy-type inequalities for $(-\Delta)^\alpha$ on $\ell^{2}(\Z)$ when $\alpha\geq1/2$, completing the picture of positive powers of the discrete Laplacian on $\Z$.

\medskip
4) \textbf{Fractional Laplacians in $L^{2}(0,\infty)$.} 
The subcriticality of the polyharmonic operators in $L^2(0,\infty)$ due to the Birman inequalities~\cite{Birman-1961,Glazman-1965,Gesztesy-Littlejohn-Michael-Wellman-2018} mentioned in point 2) are further complemented by inequalities $(-\Delta)^\alpha \ge V^{{\rm BD},\alpha}$ proved for fractional powers $\alpha\in(0,1)$ in~\cite{Bogdan-Dyda-2011}, where 
\begin{equation}
 V^{{\rm BD},\alpha}(x):=\frac{c_{\alpha}}{x^{2\alpha}}
\label{eq:V_BD}
\end{equation}
with a constant $c_{\alpha}\geq0$; see~\cite{Bogdan-Dyda-2011} for an explicit formula. The constant $c_{\alpha}$ is positive  if $\alpha\neq1/2$, implying $(-\Delta)^\alpha$ to be subcritical for $\alpha\in(0,1)\setminus\{1/2\}$.

\medskip
5) \textbf{Fractional Laplacians in $L^{2}(\R)$.} 
The continuous setting on the full line resembles its discrete analogue on $\Z$. Indeed, for fractional powers $\alpha \in (0,1/2)$, an inequality $(-\Delta)^\alpha \ge V^{{\rm He},\alpha}$ holds in $L^2(\R)$ with the weight
\begin{equation}
	V^{{\rm He},\alpha} (x) := 4^\alpha \frac{\Gamma^2\left(\frac{1+2\alpha}{4}\right)}{\Gamma^2\left(\frac{1-2\alpha}{4}\right)}\frac{1}{x^{2\alpha}},
\end{equation}
see~\cite{Herbst-1977} and also~\cite{Beckner-1995,Yafaev-1999}. This is in line with the discrete setting on $\ell^2(\Z)$, where the leading term in the asymptotic expansion of~\eqref{eq:V_CR} as $n \to \infty$ is equal to the discrete analogue of $V^{{\rm He},\alpha}$.

\medskip
6) \textbf{Powers of Laplacians in the higher dimensional setting.} An asymptotic behavior of optimal constants in the discrete Hardy and Rellich inequalities on $\Z^{d}$ is studied in~\cite{Gupta-2023} for $d\to\infty$. For further numerous research works related to fractional Laplacians defined on various subspaces of $L^{2}(\Omega)$ on open domains $\Omega\subset\R^{d}$ and general dimension $d\geq1$, we refer to the recent review~\cite{Frank-2018} and references therein.

\section{Preliminaries}\label{sec:prelim}

\subsection{Powers of the discrete Laplacian on $\ell^{2}(\N)$}

We give more details on the general properties of the discrete Laplacian and its positive powers on $\ell^{2}(\N)$, mainly their diagonalization and matrix representation.

First, one can employ basic properties of Chebyshev polynomials of the second kind $U_{n}$ to diagonalize the discrete Laplacian $-\Delta$. Recall that the sequence of Chebyshev polynomials $U_{n}$ is determined by the recurrence
\begin{equation}
 U_{n+1}(x)-2x U_{n}(x)+U_{n-1}(x)=0, \quad n\in\N,
\label{eq:recur_cheb}
\end{equation}
with the initial setting $U_{0}(x):=1$ and $U_{1}(x):=2x$; we refer the reader to~\cite[Sec.~10.11]{Erdelyi_vol2} for general properties of Chebyshev polynomials. Further, the set of functions $\{\sqrt{2/\pi}\,U_{n} \mid n\in\N_{0}\}$ forms an orthonormal basis in the Hilbert space $L^2 ((-1,1), \sqrt{1-x^2} \, \dd x)$. Therefore, the mapping~$\mathcal{U}$ defined as
\begin{equation}\label{eq:def_Uu_b}
	\mathcal U  e_n  := \sqrt \frac{2}{\pi}\, U_{n-1}, \quad n \in \N,
\end{equation}
where $e_{n}$ is the $n$-th vector of the standard basis of $\ell^{2}(\N)$, extends to a unitary operator $\mathcal U : \ell^2(\N) \to L^2 ((-1,1), \sqrt{1-x^2} \, \dd x)$.
With the aid of~\eqref{eq:recur_cheb}, it is straightforward to verify that 
\begin{equation}\label{eq:laplacian_diagonalized}
	\Uu(-\Delta)\Uu^{-1} = M_{2(1-x)},
\end{equation}
where $M_{f(x)}$ denotes the multiplication operator by a measurable function $f$ in $L^2 \big((-1,1), \sqrt{1-x^2} \, \dd x\big)$. 

From this observation, the spectral representation of $(-\Delta)^{\alpha}$ readily follows. Moreover, one can also compute the matrix representation of $(-\Delta)^{\alpha}$ with respect to the standard basis of $\ell^{2}(\N)$ which turns out to be a particular Hankel plus Toeplitz matrix. The formula for the matrix elements of $(-\Delta)^{\alpha}$ is of no explicit use is this paper but can be of independent interest.

\begin{proposition}
 Let $\alpha>0$. Then
 \[ 
 (-\Delta)^{\alpha} = \Uu^{-1}M_{2^{\alpha}(1-x)^{\alpha}}\Uu.
 \]
 Further, for $m,n\in\N$, the matrix entries of $(-\Delta)^{\alpha}$ read
 \begin{equation}
 (-\Delta)^{\alpha}_{m,n}=(-1)^{m+n}\left[\binom{2\alpha}{\alpha+m-n}-\binom{2\alpha}{\alpha+m+n}\right],
 \label{eq:frac_lapl_matrix}
 \end{equation}
 where the generalized binomial number is defined by the formula
 \begin{equation}\label{binomial}
  \binom{a}{b}:=\frac{\Gamma(a+1)}{\Gamma(b+1)\Gamma(a-b+1)}.
 \end{equation} 
 (Recall that the reciprocal Gamma function is an entire function vanishing at the points $0,-1,-2,\dots$.)
\end{proposition}

\begin{proof}
The first claim follows readily from~\eqref{eq:laplacian_diagonalized} and the functional calculus for self-adjoint operators.

Furthermore, the first claim implies that for the matrix entries $(-\Delta)^{\alpha}_{m,n}=\langle e_{m},(-\Delta)^{\alpha}e_{n}\rangle$ we have the integral representation
$$
 (-\Delta)^{\alpha}_{m,n}=\frac{2^{\alpha+1}}{\pi}\int_{-1}^{1} (1-x)^{\alpha}U_{m-1}(x)U_{n-1}(x) \sqrt{1-x^2} \dd x, \quad m,n\in\N.
$$
Formula~\eqref{eq:frac_lapl_matrix} follows from an explicit calculation of the above integral which is postponed to the Appendix, see Lemma~\ref{lem:cheb_id1}.
\end{proof}

As an immediate corollary of the last proposition, we state an integral representation for the Green kernel of $(-\Delta)^{\alpha}$.

\begin{corollary}
 Let $\alpha>0$. Then we have
 \begin{equation}\label{eq:Green.alpha}
 ((-\Delta)^\alpha - \lambda)^{-1}_{m,n} = \frac2\pi \int_{-1}^1 \frac{U_{m-1}(x) U_{n-1} (x)}{2^\alpha(1-x)^\alpha - \lambda}  \sqrt{1-x^2} \dd x
\end{equation}
for all $m,n\in\N$ and $\lambda \notin [0,4^{\alpha}]$.
\end{corollary}

\begin{remark}\label{rem:neg.alp}
Although negative powers $\alpha$ are not in the scope of the current paper, we remark on a possible extension of~\eqref{eq:frac_lapl_matrix} to $\alpha<0$, in which case $(-\Delta)^{\alpha}$ is an unbounded operator. Investigating the convergence of the resulting integrals, one sees that the Chebyshev polynomials $U_{n}$ belong to the domain or form domain of $M_{2^{\alpha}(1-x)^{\alpha}}$, respectively, if and only if $\alpha>-3/4$ or $\alpha>-3/4$. The same conditions thus hold for the standard basis vectors in $\ell^2(\N)$ to lie in the domain or form domain of $(-\Delta)^{\alpha}$. Formula~\eqref{eq:frac_lapl_matrix}  remains valid even for $\alpha>-3/2$ with the left-hand side interpreted as the corresponding quadratic form. For the apparent singularities $\alpha=-1/2$ and $\alpha=-1$, the right-hand side of~\eqref{eq:frac_lapl_matrix} is to be understood as the respective limit
\begin{align}
	(-\Delta)^{-1/2}_{m,n} & = \frac{(-1)^{m+n}}{2} \left[ \frac{ \psi \left(\frac12+m+n\right) + \psi \left(\frac12-m-n\right) }{ \Gamma \left(\frac12+m+n\right) \Gamma \left(\frac12-m-n\right)} - \frac{\psi \left(\frac12+m-n\right) + \psi \left(\frac12-m+n\right)}{\Gamma \left(\frac12+m-n\right) \Gamma \left(\frac12-m+n\right)} \right], \\
	(-\Delta)^{-1}_{m,n} & = \min(m,n),
\end{align}
where $\psi:=\Gamma'/\Gamma$ is the digamma function.
\end{remark}

\subsection{Uniform bounds on the Green kernel}

An important ingredient to our proof of Theorem~\ref{thm:main1} is a bound on the modulus of the Green kernel of $(-\Delta)^{\alpha}$ which is uniform in the spectral parameter. We prove two such bounds. First, a rather rough  but sufficient bound for the proof of subcriticality of $(-\Delta)^{\alpha}$ for $\alpha\in(0,3/2)$. Second, a refined estimate which will be used in the proof of Theorem~\ref{thm:main2}.

\begin{lemma}\label{lem:unif.bound}
Let $\alpha \in (0,3/2)$. Then there exists a constant $C_\alpha >0$ such that, for all $m,n\in\N$ and $\lambda<0$, we have 
\begin{equation}\label{eq:BS.bound}
	|((-\Delta)^\alpha - \lambda)^{-1}_{m,n}| \le C_\alpha mn.
\end{equation}
\end{lemma}

\begin{proof}
Observe that the modulus of the integrand in the integral representation~\eqref{eq:Green.alpha} is an increasing function of $\lambda<0$. Therefore we may estimate it from above by taking $\lambda=0$. The resulting integral remains convergent due to the assumption $\alpha \in (0,3/2)$. This reasoning yields the upper estimate
\[
	|((-\Delta)^\alpha - \lambda)^{-1}_{m,n}|  \le \frac{2}{\pi} \int_{-1}^1 \frac{|U_{m-1}(x) U_{n-1}(x)|}{2^\alpha (1-x)^\alpha - \lambda} \sqrt{1-x^2} \dd x \le C_\alpha \|U_{m-1}\|_{\infty} \|U_{n-1}\|_{\infty}
\]
with the positive constant 
\begin{equation}\label{eq:Green.unif.C}
	C_\alpha := \frac1{2^{\alpha-1}\pi} \int_{-1}^1\frac{\sqrt{1-x^2}}{(1-x)^\alpha} \dd x < \infty,
\end{equation}
and where 
\[
 \|U_{n}\|_{\infty}:=\max_{x\in[-1,1]}|U_{n}(x)|.
\]

We conclude the proof by showing that 
\[
\|U_n\|_{\infty}\leq n+1
\]
for all $n\in\N_{0}$, where the above is actually an equality since $U_{n}(1)=n+1$. Using identity ~\cite[Equ.~(2), Sec.~10.11]{Erdelyi_vol2}
\begin{equation}
 U_{n}(\cos\theta)=\frac{\sin(n+1)\theta}{\sin\theta},
\label{eq:U_n_cos}
\end{equation}
we obtain the expression
\[
 U_{n}(\cos\theta)=\frac{e^{\ii(n+1)\theta}-e^{-\ii(n+1)\theta}}{e^{\ii\theta}-e^{-\ii\theta}}=e^{-\ii n\theta}\sum_{k=0}^{n}e^{2\ii k\theta},
\]
from which we immediately deduce that
\[
|U_{n}(\cos\theta)|\leq\sum_{k=0}^{n}1=n+1
\]
for all $\theta\in(0,\pi)$ and $n\in\N_{0}$. The proof is complete.
\end{proof}

\begin{lemma}\label{lem:unif.bound.refined}
Let $\alpha \in (0,3/2)$. Then for all $m,n\in\N$ and $\lambda<0$ we have
\begin{equation}\label{eq:Green.bound.In}
	|((-\Delta)^\alpha - \lambda)^{-1}_{m,n}| \le \frac{1}{2\pi}\frac{\Gamma^{2}(\alpha)}{\Gamma(2\alpha)} \sqrt{g_m (\alpha)} \sqrt{g_n(\alpha)},
\end{equation}
where $g_{n}$ is as in~\eqref{eq:def_g_n} and~\eqref{eq:g_n_1/2_1}.
\end{lemma}

\begin{proof}
Since $\alpha\in(0,3/2)$, the integral in formula~\eqref{eq:Green.alpha} is convergent for $\la = 0$. Therefore, with fixed $m,n \in \N$ and $\la < 0$, we can estimate the Green kernel as follows
\[
	|((-\Delta)^\alpha - \lambda)^{-1}_{m,n}| \le \frac1{2^{\alpha-1}\pi} \int_{-1}^1 \frac{|U_{m-1}(x) U_{n-1}(x)|} {(1-x)^{\alpha}} \sqrt{1-x^2} \dd x \le \frac1{2^{\alpha-1}\pi} \sqrt{I_m (\alpha)} \sqrt{I_n(\alpha)},
\]
where we  used the Cauchy--Schwarz inequality and defined
\begin{equation}\label{eq:def.In}
	I_n (\alpha):= \int_{-1}^1 \frac{U_{n-1}^2 (x)}{(1-x)^{\alpha}} \sqrt{1-x^2} \dd x.
\end{equation}
The rest of the proof follows from a formula for $I_{n}(\alpha)$ in the Appendix, see Lemma~\ref{lem:cheb_id2}, where it is proven that 
\[
I_{n}(\alpha)=2^{\alpha-2}\,\frac{\Gamma^{2}(\alpha)}{\Gamma(2\alpha)}\,g_{n}(\alpha)
\]
for any $\alpha\in(0,3/2)$ and $n\in\N$ (see therein also the limiting formulas for $\alpha = 1/2$ and $\alpha = 1$).
\end{proof}

\section{Proofs of Theorems~\ref{thm:main1}, \ref{thm:main2}, and \ref{thm:main3}}\label{sec:proofs}

Our method relies on the Birman--Schwinger principle, see e.g.~\cite{Hansmann-Krejcirik-2022}, which allows to relate both criticality and subcriticality to the behavior of the Green kernel~\eqref{eq:Green.alpha} as the spectral parameter $\lambda$ approaches the spectrum. The uniform bounds from Lemmas~\ref{lem:unif.bound} and~\ref{lem:unif.bound.refined} guarantee a finite limit of the Green kernel as $\lambda\to0^-$. This leads to a Hardy-type inequality and thus to the subcriticality of $(-\Delta)^{\alpha}$ for $\alpha\in(0,3/2)$. On the other hand, if $\alpha\geq3/2$, the diagonal entries of the Green kernel have a singularity as $\lambda\to0^-$, which results in the criticality of $(-\Delta)^{\alpha}$. 

\subsection{Proof of Theorem~\ref{thm:main1}}\label{sec:proof1}

Recall that we always assume $\alpha>0$. The statement of Theorem~\ref{thm:main1} is proven in two steps:

\begin{enumerate}
\item[(i)] If $\alpha<3/2$, then $(-\Delta)^{\alpha}$ is subcritical.
\item[(ii)] If $\alpha\geq 3/2$, then $(-\Delta)^{\alpha}$ is critical.
\end{enumerate}

\begin{proof}[Step~(i): Proof of the subcriticality in Theorem~\ref{thm:main1}]
Suppose $\alpha \in (0,3/2)$. Consider the potential
\[
 V_{n}:=\frac{\gamma}{n^{4}}, \quad n\in\N,
\]
where $\gamma>0$. Since $V$ is compact (even trace class) the spectrum of $(-\Delta)^{\alpha}-V$ below $0$ can contain only eigenvalues. With this particular choice of $V$, we show that there exists a sufficiently small $\gamma>0$ such that the operator norm of the Birman--Schwinger operator
\begin{equation}\label{eq:K.la.subcrit}
	K(\lambda):= - V^\frac12 ((-\Delta)^\alpha - \lambda)^{-1} V^\frac12
\end{equation}
fulfills 
\begin{equation}
 \sup_{\lambda<0}\|K(\lambda)\|<1.
\label{eq:sup_K_leq_1}
\end{equation}
By the Birman--Schwinger principle, it follows that there exists no negative eigenvalue of $(-\Delta)^{\alpha}-V$. Consequently, $(-\Delta)^{\alpha}\geq V$ and $(-\Delta)^{\alpha}$ is therefore subcritical. 

By means of Lemma~\ref{lem:unif.bound}, we derive the following bound on the Hilbert Schmidt (and thus operator) norm of $K(\lambda)$,
\begin{equation}
	\|K(\lambda)\| \le \|K(\lambda)\|_{\rm HS} = \left(\sum_{m=1}^{\infty} \sum_{n=1}^{\infty} |K_{m,n}(\lambda)|^2\right)^{\!1/2} \le C_\alpha\sum_{n=1}^{\infty}n^{2}V_{n}=C_{\alpha}\sum_{n=1}^{\infty}\frac{\gamma}{n^{2}}
\label{eq:K_norm_estim}
\end{equation}
for all $\lambda<0$. Thus, for any $0<\gamma<6/(\pi^{2}C_{\alpha})$, inequality~\eqref{eq:sup_K_leq_1} holds true.
\end{proof}

For the proof of Step (ii), we will need two auxiliary results.

\begin{lemma}\label{lem:sing}
Let $\alpha \ge 3/2$. Then, for all $n\in\N$, we have
\begin{equation}\label{eq:Green.sing}
	\lim_{\lambda \to 0^-} ((-\Delta)^\alpha - \lambda)_{n,n}^{-1} = + \infty.
\end{equation}
\end{lemma}

\begin{proof}
Taking the formal limit $\la \to 0^-$ under the integral in~\eqref{eq:Green.alpha} produces a singularity in the integrand at $x=1$. The strategy is to split off an interval touching this singularity where the integrand is strictly positive (and one can thus easily show this portion of the integral to diverge). The remaining term can be easily bounded uniformly in $\lambda$. For this partition, the sign of the Chebyshev polynomials $U_n$ is of interest. As we know from the proof of Lemma~\ref{lem:unif.bound},
\begin{equation}\label{eq:Cheb.max}
	\max_{x \in [-1,1]} |U_{n-1} (x)| = U_{n-1}(1) = n, \quad n\in\N.
\end{equation}
Therefore we can pick a point larger than the largest zero of $U_{n-1}$, e.g.
\begin{equation}
	\max \left \{x \in \R \, : \, U_{n-1}(x) = 0 \right\} < a_{n} := \cos \left(\frac{\pi}{n+1} \right) \in (0,1), \quad n \in  \N,
\end{equation}
and we have
\begin{equation}\label{eq:min.Un}
	\min_{x \in [a_n,1]} U_{n-1}(x) > 0 , \quad n \in  \N.
\end{equation}

For arbitrary fixed $n \in \N$ and a spectral parameter $\lambda <0$, we start estimating~\eqref{eq:Green.alpha} with $m=n$ from below. Splitting the interval of integration at $a_n$,  using the triangle inequality and omitting $\la < 0$ under the second integral we arrive at
\begin{equation}\label{eq:Green.lower.bd}
	((-\Delta)^\alpha - \lambda)^{-1}_{n,n}  \ge  \frac2\pi \int_{a_n}^1 \frac{U_{n-1}^2(x)}{2^\alpha (1-x)^\alpha - \lambda}  \sqrt{1-x^2} \dd x - \frac2\pi \int_{-1}^{a_n} \frac{U_{n-1}^2(x)}{2^\alpha (1-x)^\alpha } \sqrt{1-x^2} \dd x.
\end{equation}
Since the second term is now independent of $\lambda$ and the integral converges, to prove~\eqref{eq:Green.sing} it suffices to show that the first integral in~\eqref{eq:Green.lower.bd} tends to infinity as $\lambda \to 0^-$. Using that $\alpha \ge 3/2$ and $a_n \ge 0$, we have
\begin{equation}
	(1-x)^{\alpha} \le (1-x)^{\frac32}
\end{equation}
for all  $x \in [a_n, 1] \subset [0,1]$.
We therefrom conclude a lower estimate on the first integral in~\eqref{eq:Green.lower.bd} as follows
\begin{equation}\label{eq:Green.lower.bd2}
	\frac2\pi \int_{a_n}^1 \frac{U_{n-1}^2(x)}{2^\alpha (1-x)^\alpha - \lambda}  \sqrt{1-x^2} \dd x \ge \sqrt{1+a_n}\min_{x \in [a_n,1]}\left(U_{n-1}^2 (x)\right) \int_{a_n}^1 \frac{\sqrt{1-x}}{2^\alpha (1-x)^{3/2} - \lambda} \dd x.
\end{equation}
Hence it suffices to check that the integral on the right-hand side tends to infinity as $\lambda\to0^{-}$. This is the case indeed as, by the monotone convergence, one has
\[
\lim_{\lambda\to0^{-}}\int_{a_n}^1 \frac{\sqrt{1-x}}{2^\alpha (1-x)^{3/2} - \lambda} \dd x=\frac{1}{2^\alpha}\int_{a_n}^1 \frac{\dd x}{1-x}=\infty.
\]
\end{proof}

We next use the above lemma to show that, when $\alpha\geq3/2$, for any arbitrarily small localized perturbation of $(-\Delta)^{\alpha}$ a unique eigenvalue emerges from the bottom of the spectrum.
We denote by
\[
 \delta_{n}:=\langle e_{n}, \,\cdot\,\rangle e_{n}
\]
the delta potential localized at $n\in\N$.

\begin{lemma}\label{lem:ev}
Let $\alpha \ge 3/2$. Then for any $n \in \N$ and $c>0$ the operator $(-\Delta)^{\alpha}-c \delta_n$ has a unique negative eigenvalue.
\end{lemma}

\begin{proof}
Fix $c > 0$ and $n \in \N$. Since $\delta_n$ is a projection, one has $\delta_{n}=\delta_{n}^{2}$ and the corresponding Birman--Schwinger operator reads
\begin{equation}\label{eq:K.la}
	K(\lambda) := - c \delta_n ((-\Delta)^\alpha - \lambda)^{-1} \delta_n, \quad \la \in \C \setminus [0,4^\alpha].
\end{equation}
For all $\la \in \C \setminus [0,4^\alpha]$, the Birman--Schwinger principle gives the equivalence
\begin{equation}\label{eq:BS}
	\lambda \in \sigma_{\rm p} ((-\Delta)^\alpha - c\delta_n) \quad \iff \quad -1 \in \sigma_{\rm p} \left(K(\lambda)\right).
\end{equation}
Since $K(\lambda)$ is a rank one operator with the single non-zero eigenvalue
\begin{equation}\label{eq:def.mu}
	\mu (\lambda) := - c  ((-\Delta)^\alpha - \lambda)^{-1}_{n,n},
\end{equation}
equivalence~\eqref{eq:BS} can be rewritten as 
\begin{equation}\label{eq:mu.BS}
	\lambda \in \sigma_{\rm p} ((-\Delta)^\alpha - c\delta_n) \quad \iff \quad \mu(\lambda)=-1.
\end{equation}
The strategy is to show that there exists a unique $\la=\lambda_{n}(c) <0$ such that $\mu(\lambda)=-1$. For this it is sufficient to verify that $\mu$, as a function on $(-\infty, 0)$, has the following properties:
\begin{enumerate}[label=(\alph*)]
\item $\mu$ is strictly decreasing on $(-\infty,0)$,
\item $\mu$ is continuous on $(-\infty,0)$,
\item $\mu(\lambda)\to0$ as $\lambda\to-\infty$,
\item $\mu(\lambda)\to-\infty$ as $\lambda\to0^{-}$.
\end{enumerate}

Property~(a) is immediate from the integral representation
\begin{equation}
\mu(\lambda)=-\frac{2c}{\pi}\int_{-1}^1 \frac{U_{n-1}^{2}(x)}{2^\alpha(1-x)^\alpha - \lambda}  \sqrt{1-x^2} \dd x,
\label{eq:mu_int_repre}
\end{equation}
see~\eqref{eq:Green.alpha}. Since the resolvent $((-\Delta)^{\alpha}-\lambda)^{-1}$ is an analytic (operator-valued) function of $\lambda$ on $\C \setminus [0,4^\alpha]$, (b) follows. 
By monotone convergence applied to the integral in~\eqref{eq:mu_int_repre}, one easily verifies property~(c). 
Finally, (d) is a consequence of Lemma~\ref{lem:sing}.
\end{proof}

We are now ready to prove the remaining part of Theorem~\ref{thm:main1}.

\begin{proof}[Step~(ii): Proof of the criticality in Theorem~\ref{thm:main1}]
Let $\alpha \ge 3/2$. From Lemma~\ref{lem:ev} it follows that if 
\[
(-\Delta)^\alpha \ge c \delta_n
\]
with any $n\in\N$ and $c\geq0$, then necessarily $c=0$. Suppose that $V\geq0$ is a given bounded potential such that $(-\Delta)^\alpha \ge V$. Then, for all $n\in\N$, one has $V\geq V_{n}\delta_{n}$ and thus $(-\Delta)^\alpha \ge V_{n}\delta_{n}$. By the observation above, it follows that $V_{n}=0$. Since $n\in\N$ is arbitrary, we conclude $V=0$.
\end{proof}

\subsection{A comment on the subcriticality of $(-\Delta)^{\alpha}$ from above}\label{subsec:crit_from_above}
Recall that $\sigma((-\Delta)^{\alpha})=[0,4^{\alpha}]$ for any $\alpha>0$. Analogously to the usual notion of criticality of $(-\Delta)^{\alpha}$, one may ask about the stability of the upper bound $4^{\alpha}$ when $(-\Delta)^{\alpha}$ is perturbed by a bounded potential $V\geq0$. In other words, we may investigate the criticality or subcriticality of the operator $4^{\alpha}-(-\Delta)^{\alpha}$. It turns out that this operator is subcritical for all $\alpha>0$, which can be seen from essentially the same arguments as in the proof of Theorem~\ref{thm:main1}.

\begin{proposition}
The operator $4^{\alpha}-(-\Delta)^{\alpha}$ is subcritical for all $\alpha>0$.
\end{proposition}

\begin{proof}
 Fix any $\alpha >0$. First, from~\eqref{eq:Green.alpha} we deduce the integral representation of the Green function
 \[
  \left(4^{\alpha}-(-\Delta)^{\alpha}-\lambda\right)_{m,n}^{-1}=\frac2\pi \int_{-1}^1 \frac{U_{m-1}(x) U_{n-1} (x)}{4^{\alpha}-2^\alpha(1-x)^\alpha - \lambda}  \sqrt{1-x^2} \dd x
 \]
 for any $\lambda\notin[0,4^{\alpha}]$ and $m,n\in\N$. Next, similarly as in Lemma~\ref{lem:unif.bound}, we deduce the uniform bound
 \[
 |\left(4^{\alpha}-(-\Delta)^{\alpha}-\lambda\right)_{m,n}^{-1}|\leq C_{\alpha}mn
 \]
 for all $\la<0$, with the positive constant 
 \[ 
 C_\alpha := \frac2 \pi \int_{-1}^1 \frac{\sqrt{1-x^2}}{4^\alpha - 2^\alpha (1-x)^\alpha} \dd x.
 \]
 The only difference to the proof of the subcriticality of $(-\Delta)^{\alpha}$ is that the above integral is always finite regardless the value of $\alpha>0$. Indeed, its integrand is a continuous function on $(-1,1]$ which equals $O(\sqrt{x+1})$ as $x\to-1^{+}$. The rest of the proof is analogous to step (i) in the proof of Theorem~\ref{thm:main1}.
\end{proof}

\subsection{Proof of Theorem~\ref{thm:main2}}\label{sec:proof2}

As a byproduct of our method, in the proof of the subcriticality in Theorem~\ref{thm:main1} we already obtain a (rather rough) Hardy inequality for $(-\Delta)^{\alpha}$. This is done using the bound from Lemma~\ref{lem:unif.bound} when estimating the norm of the Birman--Schwinger operator, see~\eqref{eq:K_norm_estim}. We proceed similarly with the refined bound in Lemma~\ref{lem:unif.bound.refined} to prove Theorem~\ref{thm:main2}.

\begin{proof}[Proof of Theorem~\ref{thm:main2}]
 Let $\alpha\in(0,3/2)$. Suppose first that a potential $V\geq0$ satisfies condition~\eqref{eq:V_suff_cond} with the strict inequality, i.e.~that
 \begin{equation}
 \sum_{n=1}^{\infty}g_{n}(\alpha)V_{n}< 2\pi\frac{\Gamma(2\alpha)}{\Gamma^{2}(\alpha)}.
\label{eq:V_suff_cond_strict}
\end{equation}
Using Lemma~\ref{lem:unif.bound.refined}, we estimate the norm of the corresponding Birman--Schwinger operator, cf.~\eqref{eq:K.la.subcrit} and~\eqref{eq:K_norm_estim}, as follows
\[
 \|K(\lambda)\|\leq\left(\sum_{m=1}^{\infty} \sum_{n=1}^{\infty} |K_{m,n}(\lambda)|^2\right)^{\!1/2}\! =\left(\sum_{m=1}^{\infty} \sum_{n=1}^{\infty} V_{m}\left|\left((-\Delta)^{\alpha}-\lambda\right)_{m,n}^{-1}\right|^{2}V_{n}\right)^{\!1/2} \! \le \frac{1}{2\pi}\frac{\Gamma^{2}(\alpha)}{\Gamma(2\alpha)}\sum_{n=1}^{\infty}V_{n}g_{n}(\alpha)
\]
for all $\lambda<0$. Hence, assumption~\eqref{eq:V_suff_cond_strict} implies $\|K(\lambda)\|<1$ for all $\lambda<0$, and thus further $(-\Delta)^{\alpha}\geq V$ by the Birman--Schwinger principle.

Suppose now that $V\geq0$ satisfies~\eqref{eq:V_suff_cond}. We introduce the auxiliary potentials $V(q):=q V$ for all $q\in(0,1)$. Then~\eqref{eq:V_suff_cond_strict} holds for $V(q)$ and therefore 
\[
(-\Delta)^{\alpha}\geq V(q) 
\]
for all $q\in(0,1)$ by the first part of this proof. Since $V(q)\to V$ converges strongly as $q\to1^{-}$, the above inequality in sense of forms remains valid in the limit and we conclude that $(-\Delta)^{\alpha}\geq V$.
\end{proof}

\subsection{Proof of Theorem~\ref{thm:main3}}\label{sec:proof3}

The idea of the proof is to estimate the function $g_{n}(\alpha)$, given by~\eqref{eq:def_g_n} and~\eqref{eq:g_n_1/2_1}, from above and then apply Theorem~\ref{thm:main2}. Investigating its asymptotic behavior as $n\to\infty$, one finds that 
\[
g_{n}(\alpha)=\begin{cases} 
				O(1)& \mbox{ if } \alpha\in(0,1/2),\\
				O(\ln n) & \mbox{ if } \alpha=1/2,\\
				O(n^{2\alpha-1}) & \mbox{ if } \alpha\in(1/2,3/2).\\
			  \end{cases}
\]
Therefore, the asymptotic behavior of an admissible Hardy weight is expected to depend on $\alpha$, according to the convergence of the sum in~\eqref{eq:V_suff_cond}. In the upcoming proof of Theorem~\ref{thm:main3}, the three cases depending on the value of $\alpha$ are discussed separately. We thereby use the following auxiliary inequalities.

\begin{lemma}\label{lem:g_n_bounds}
For all $n\in\N$, we have
\[
g_{n}(\alpha)\leq\begin{cases} 
				\tan(\alpha\pi)& \mbox{ if } \alpha\in(0,1/2),\\
				\frac{3+\ln n}{\pi} & \mbox{ if } \alpha=1/2,\\
				D_{\alpha}n^{2\alpha-1} & \mbox{ if } \alpha\in(1/2,3/2),\\
			  \end{cases}
\]
with the positive constant
\[
 D_{\alpha}:=\left(\chi_{(1,3/2)}(\alpha)+\frac{\alpha(1+\alpha)2^{2\alpha-1}}{(\alpha-1)(2-\alpha)^{2\alpha}}\right)\tan(\alpha\pi),
\]
where $\chi$ denotes the characteristic function. (For $\alpha=1$, the limiting value $D_{1}=4\pi$  is to be taken.)
\end{lemma}

\begin{proof}
We distinguish several cases depending on the value of $\alpha$.

\emph{Case $\alpha\in(0,1/2)$:} It suffices to notice that $(\alpha)_{k}/(1-\alpha)_{k}>0$ for all  $\alpha\in(0,1/2)$ and $k\in\N$ (this cannot be improved since the fraction is decreasing in $k$ with zero limit as $k\to\infty$). Hence, it follows readily from definition~\eqref{eq:def_g_n} that 
\[
 g_{n}(\alpha)\leq\tan(\alpha\pi).
\]

\emph{Case $\alpha=1/2$:} It is not difficult to verify that
\[
 \sum_{j=1}^{2n}\frac{1}{2j-1}\leq \frac{3}{2}+\frac{1}{2}\ln n
\]
for all $n\in\N$. The claimed estimate for $g_{n}(1/2)$ then follows from~\eqref{eq:g_n_1/2_1}.

\emph{Case $\alpha=1$:} Immediate from~\eqref{eq:g_n_1/2_1}.

\emph{Case $\alpha\in(1/2,3/2)\setminus\{1\}$:} Since in this case $2 \alpha -1 >0$ and $\tan (\alpha \pi) <0$ when $\alpha \in (1/2, 1)$, from~\eqref{eq:def_g_n} we obtain the upper bound
\begin{equation}
g_{n}(\alpha)\leq\left(\chi_{(1,3/2)}(\alpha)n^{2\alpha-1}+\left|\frac{(\alpha)_{2n}}{(1-\alpha)_{2n}}\right|\right)|\tan(\alpha\pi)|.
\label{eq:g_n_first_ineq_inproof}
\end{equation}
To arrive at the claim, we need to estimate the absolute value of the ratio of Pochhammer symbols above. Observing that 
\[
\frac{(\alpha)_{2n}}{(1-\alpha)_{2n}}=\frac{\alpha(1+\alpha)}{(1-\alpha)(2-\alpha)}\frac{(2+\alpha)_{2n-2}}{(3-\alpha)_{2n-2}}
\]
and that the second fraction on the right-hand side is always positive for $\alpha\in(1/2,3/2)$ and $n\in\N$, we proceed by
\begin{align}
\ln\frac{(2+\alpha)_{2n-2}}{(3-\alpha)_{2n-2}}&=\ln\prod_{j=1}^{2n-2}\left(1+\frac{2\alpha-1}{j+2-\alpha}\right)=\sum_{j=1}^{2n-2}\ln\left(1+\frac{2\alpha-1}{j+2-\alpha}\right)\leq\sum_{j=1}^{2n-2}\frac{2\alpha-1}{j+2-\alpha}\nonumber\\
&\leq\int_{0}^{2n-2}\frac{2\alpha-1}{j+2-\alpha}\dd j=(2\alpha-1)\ln\frac{2n-\alpha}{2-\alpha}.
\label{eq:estim_ratio_poch_inproof}
\end{align}
Therefore 
\[
 \left|\frac{(\alpha)_{2n}}{(1-\alpha)_{2n}}\right|\leq\frac{\alpha(1+\alpha)}{|1-\alpha|(2-\alpha)}\left(\frac{2n-\alpha}{2-\alpha}\right)^{2\alpha-1}\leq\frac{\alpha(1+\alpha)2^{2\alpha-1}}{|1-\alpha|(2-\alpha)^{2\alpha}}\,n^{2\alpha-1}
\]
for all $\alpha\in(1/2,3/2)\setminus\{1\}$ and $n\in\N$. When this is applied in~\eqref{eq:g_n_first_ineq_inproof}, noticing that $1-\alpha$ and $\tan(\alpha \pi)$ have the same sign, we obtain the claim. 
\end{proof}

\begin{proof}[Proof of Theorem~\ref{thm:main3}]
\emph{Case $\alpha\in(0,1/2)$:} Fix $\varepsilon>0$ and consider $V$ given by~\eqref{eq:power-form_Hardy_weight}, i.e.
\[
 V_{n}=\frac{\gamma}{n^{1+\varepsilon}}, \quad n\in\N,
\]
where the constant $\gamma>0$ is to be specified so that the condition~\eqref{eq:V_suff_cond} of Theorem~\ref{thm:main2} holds. Using Lemma~\ref{lem:g_n_bounds}, we find 
\[
 \sum_{n=1}^{\infty}g_{n}(\alpha)V_{n}\leq\gamma\tan(\alpha\pi)\sum_{n=1}^{\infty}\frac{1}{n^{1+\varepsilon}}=\gamma\tan(\alpha\pi)\zeta(1+\varepsilon),
\]
where $\zeta$ is the Riemann zeta function. Thus, taking
\[
 \gamma=\gamma(\alpha,\varepsilon):=\frac{2\pi\Gamma(2\alpha)}{\Gamma^{2}(\alpha)\tan(\alpha\pi)\zeta(1+\varepsilon)},
\]
condition~\eqref{eq:V_suff_cond} is fulfilled and $(-\Delta)^{\alpha}\geq V$ by Theorem~\ref{thm:main2}.

\emph{Case $\alpha=1/2$:} We proceed analogously as in the previous case. Using the respective estimate from Lemma~\ref{lem:g_n_bounds}, we find
\[
 \sum_{n=1}^{\infty}g_{n}(\alpha)V_{n}\leq\frac{\gamma}{\pi}\sum_{n=1}^{\infty}\frac{3+\ln n}{n^{1+\varepsilon}}=\frac{\gamma}{\pi}\left[3\zeta(1+\varepsilon)-\zeta'(1+\varepsilon)\right].
\]
Consequently, if we put
\[
 \gamma=\gamma(\alpha,\varepsilon):=\frac{2\pi^{2}\Gamma(2\alpha)}{\Gamma^{2}(\alpha)\left[3\zeta(1+\varepsilon)-\zeta'(1+\varepsilon)\right]},
\]
the statement follows from Theorem~\ref{thm:main2}.

\emph{Case $\alpha\in(1/2,3/2)$:} In this case, the potential $V$ from~\eqref{eq:power-form_Hardy_weight} reads
\[
 V_{n}=\frac{\gamma}{n^{2\alpha+\varepsilon}}, \quad n\in\N.
\]
With the aid of Lemma~\ref{lem:g_n_bounds}, we find
\[
 \sum_{n=1}^{\infty}g_{n}(\alpha)V_{n}\leq\gamma D_{\alpha}\,\zeta(1+\varepsilon).
\]
Hence, by taking
\begin{align}
 \gamma=\gamma(\alpha,\varepsilon):= & \frac{2\pi\Gamma(2\alpha)}{D_{\alpha}\Gamma^{2}(\alpha)\zeta(1+\varepsilon)} \\
 = &\frac{2\pi(1-\alpha)(2-\alpha)^{2\alpha}\Gamma(2\alpha)}{(\chi_{(1,3/2)}(\alpha)(1-\alpha)(2-\alpha)^{2\alpha} + \alpha(\alpha+1)2^{2\alpha-1})\Gamma^2(\alpha)\tan(\alpha\pi)\zeta(1+\varepsilon)}, \label{eq:gamma_alpha_eps}
\end{align}
we obtain the desired inequality as a consequence of Theorem~\ref{thm:main2}. For $\alpha = 1$, the above formula is understood as the limiting value $\gamma(1,\eps) = 1/ \zeta (1+\eps)$.
\end{proof}

\section{Perturbations of the bilaplacian by localized potentials}
\label{sec:bilap}

For the discrete bilaplacian, the unique negative eigenvalue of the perturbation $\Delta^2 - c\delta_n$ described in Lemma~\ref{lem:ev} can be characterised by an implicit equation. If $n=1$, it can even be expressed fully explicitly as a function of the coupling constant $c$. Recall $U_{n}$ denotes the $n$th Chebyshev polynomial of the second kind.

\begin{proposition}\label{prop:bilap}
	For $n \in \N$ and $c>0$, the unique negative eigenvalue of $\Delta^2-c\delta_n$ is given as
	\begin{equation}
		\la_n(c) = - \frac{(1-r^2)^4}{r^2(1+r^2)^2}
	\end{equation}
	where $r$ is the unique solution of the implicit equation
	\[
	\frac{r^2}{1-r^2} \sum_{j=0}^{n-1} r^{2j}\, U_{2j} \left(\frac{2 r}{1+r^2}\right)=\frac{1}{c}
	\]
	in $(0,1)$.
	In particular, for $n=1$ it is
	\begin{equation}
		\la_1(c) = -\frac{c^4}{(c+1)(c+2)^2}.
	\end{equation}
Moreover, for $n\ge 2$ we have the asymptotic formulas
	\begin{align}
	 \lambda_{n}(c) & = -\frac{n^8 c^4}{4} \left( 1 - \frac{2n (4n^2 -1)}{3} c + \mathcal O \left(c^2\right) \right), && \hspace{-2cm} c \to0^+, \\
	 \lambda_{n}(c) & = -c \left( 1 -  \frac{6}{c} + \mathcal O \left(\frac{1}{c^2}\right)\right), && \hspace{-2cm} c \to\infty.
	\end{align}
\end{proposition}

The proof relies on the defining properties~\eqref{eq:def.mu} and~\eqref{eq:mu.BS} for $\la_n(c)$ and an explicit formula for the Green kernel of the resolvent. The latter might be of independent interests and is therefore stated in a lemma. It uses a convenient transformation of the spectral parameter. Recall that the Joukowski transform is the bijection
\begin{equation}
	\phi: \mathbb D \setminus \{0\} \to \C \setminus [-2,2], \quad \phi (z) := z + z^{-1},
\end{equation}
where $\mathbb{D}:=\{z\in\C \mid |z|<1\}$. For every $\la \notin \sigma (\Delta^2) = [0,16]$, there exist unique $0 \neq \xi, \eta \in \mathbb D$ such that
\begin{equation}\label{eq:Jouk.xi.eta}
	2 + \sqrt \la = \xi + \xi^{-1}, \quad 2 - \sqrt \la = \eta + \eta^{-1}.
\end{equation}
Notice that the resolvent formula below does not depend on the particular definition of the complex square root. Choosing another branch only exchanges the roles of $\xi$ and $\eta$, in which the formula indeed commutes.

\begin{lemma}\label{lem:Green.bilap}
	For $\la \in \C \setminus[0,16]$ and $m,n \in \N$, the Green kernel of the bilaplacian is given by
	\begin{equation}
		(\Delta^2 - \la)^{-1}_{m,n} = \frac{\xi\eta}{(1-\xi\eta)(\xi-\eta)} \left( \frac{\xi^{m+n} - \xi^{|m-n|}}{\xi-\xi^{-1}} - \frac{\eta^{m+n} - \eta^{|m-n|}}{\eta-\eta^{-1}} \right).
	\end{equation}
	Here $0 \neq \xi, \eta \in \mathbb D$ are uniquely determined by~\eqref{eq:Jouk.xi.eta}.
\end{lemma}

\begin{proof}
	Follows immediately from~\eqref{eq:Green.alpha} and Lemma~\ref{lem:bilap}.
\end{proof}

\begin{proof}[Proof of Proposition~\ref{prop:bilap}]
	Notice first that if $\la < 0$ then $\sqrt \la \in \ii \R$, such that $\xi = \overline \eta$ follows easily from~\eqref{eq:Jouk.xi.eta}. Hence, using Lemma~\ref{lem:Green.bilap} with $m=n$, the defining relations~\eqref{eq:def.mu} and~\eqref{eq:mu.BS} for $\la = \la_n(c)$ become 
	\begin{equation}
		\frac 1c = \frac{|\xi|^2}{1-|\xi|^2} \frac{1}{\im \, \xi} \im \left( \frac{\xi^{2n} - 1}{\xi-\xi^{-1}} \right) =   \frac{|\xi|^2}{1-|\xi|^2} \sum_{j=0}^{n-1} \frac{\im \, (\xi^{2j+1})}{\im \, \xi}.
	\end{equation}
	Next, from~\eqref{eq:Jouk.xi.eta} it is elementary to derive
	\begin{equation}\label{eq:re.im.xi}
		\re \, \xi = \frac{2 |\xi|^2}{1+|\xi|^2}, \quad \im^2 \xi = -\frac{\la |\xi|^4}{(1-|\xi|^2)^2},
	\end{equation}
	and thus for the cosine of the argument 
	\begin{equation}
		\cos({\rm Arg} \, \xi) =  \frac{\re \, \xi}{|\xi|} = \frac{2 |\xi|}{1+|\xi|^2}.
	\end{equation}
	Using~\eqref{eq:U_n_cos}, we can further compute
	\begin{equation}
		\frac{\im \, (\xi^{2j+1})}{\im \, \xi} = |\xi|^{2j} \frac{\sin ((2j+1) {\rm Arg} \, \xi)}{\sin ({\rm Arg}\, \xi)} = |\xi|^{2j} U_{2j} \left(\frac{2 |\xi|}{1+|\xi|^2}\right).
	\end{equation}
	Combining the above, we see that $|\xi|$ solves the equation
	\begin{equation}\label{eq:impl_func}
		\frac1c = \frac{|\xi|^2}{1-|\xi|^2} \sum_{j=0}^{n-1} |\xi|^{2j} U_{2j} \left(\frac{2 |\xi|}{1+|\xi|^2}\right).
	\end{equation}
	The dependence of $\la$ on $|\xi|$ can be expressed from~\eqref{eq:re.im.xi} as
	\begin{equation}\label{eq:lam_rel_mod_xi}
		\la = - \frac{(1-|\xi|^2)^4}{|\xi|^2(1+|\xi|^2)^2}.
	\end{equation}
	From this it is easy to see that there is a one to one correspondence between $\la <0$ and $|\xi| \in (0,1)$. Hence an occurrence of two different solutions $|\xi|$ of~\eqref{eq:impl_func} gives rise to two different eigenvalues of $\Delta^{2}+c\delta_{n}$, which would contradict Lemma~\ref{lem:ev}. Therefore there is exactly one solution $|\xi|$ of \eqref{eq:impl_func} located in $(0,1)$. The formula for $n=1$ then follows easily from~\eqref{eq:impl_func} and~\eqref{eq:lam_rel_mod_xi}.
	
	To prove the asymptotic relations, we first derive the respective expansions for $|\xi|^2 = |\xi|^2 (c)$ using the implicit function theorem in equation~\eqref{eq:impl_func}. Writing $|\xi|^2 =: s$ and observing that $U_{2j}$ contains only even degree monomials, we see that $(c,s) = (0,1)$ solves 
	\begin{equation}\label{eq:impl_F}
		F(c,s) : = c \sum_{j=0}^{n-1} s^{j+1} p_{j} \left(\frac{4 s}{(1+s)^2}\right) - (1-s) = 0,
	\end{equation}
	where $p_j$ is the polynomial of degree $j$ such that $p_j(x^2) = U_{2j} (x)$. Since $F$ is holomorphic in a neighborhood of $(0,1)$ and one can easily see that $\partial_s F (0,1) = 1$, it follows from the implicit function theorem~\cite[Thm.~7.6]{Fritzsche-Grauert-2002} that there exists a solution $s = s(c)$ of~\eqref{eq:impl_F} analytic in a neighborhood of $c = 0$. Since the solution $|\xi|$ of~\eqref{eq:impl_func} is unique for positive $c$ and tends to one as $c \to 0^+$, we have  $|\xi|^2 (c) = s(c)$ for $c>0$ in this neighborhood. 
	
	The coefficients of the sought expansion can be determined by implicit differentiation of~\eqref{eq:impl_F}, which first gives
	\begin{equation}
		0 = \frac{\dd F}{\dd c}(c,s) = \partial_c F (c,s) + \partial_s F (c,s) s'(c).
	\end{equation}
	Using~\eqref{eq:Cheb.max}, we see that $p_j(1) = U_{2j} (1) = 2j+1$ and thus $\partial_c F(0,1) = n^2$. Consequently, $s'(0) = -n^2$. Differentiating once more leads to
	\begin{equation}\label{eq:impl_diff_2}
		0 = \frac{\dd^2F}{\dd c^2}(c,s) = 2 \partial_s \partial_c F (c,s) s'(c) + \partial_s^2 F(c,s) s'(c)^2 + \partial_s F(c,s) s'' (c),
	\end{equation} 
	where we used $\partial_c^2 F \equiv 0$. Noticing also that $\partial_s^2 F (0,1) = 0$, one further computes
	\begin{equation}
		s''(0) = - 2 \frac{\partial_s \partial_c F (0,1)}{\partial_s F(0,1)} s'(0) = \frac{n^3(n+1)(4n-1)}{3}.
	\end{equation}
	In the second identity we used
	\begin{equation}
		\partial_s \partial_c F (c,s) = \sum_{j=0}^{n-1} \left[ (j+1) s^{j} p_{j} \left(\frac{4 s}{(1+s)^2}\right) + s^{j+1} p_j'\left(\frac{4 s}{(1+s)^2}\right) \frac{4 (1+s)^2 - 8s (1+s)}{(1+s)^4} \right]
	\end{equation}
	and $p_j(1) = 2j+1$ to calculate the mixed derivative at $(0,1)$. From Taylor's theorem we now obtain the asymptotic expansion
	\begin{equation}
		|\xi|^2 (c) = 1 - n^2 c + \frac{n^3(n+1)(4n-1)}{6} c^2 + \mathcal O \left(c^3\right), \quad c \to 0^+.
	\end{equation}

	For the asymptotics around infinity, we set $d := 1/c$ in~\eqref{eq:impl_func} and analyze the implicit equation
	\begin{equation}
		G(d,s) := \frac{1}{1-s}\sum_{j=0}^{n-1} s^{j+1} p_{j} \left(\frac{4 s}{(1+s)^2}\right) - d = 0
	\end{equation}
	around its solution $(d,s) = (0,0)$. Clearly, $\partial_d G (d,s) \equiv -1$ and $\partial_d^2 G (d,s) \equiv\partial_d \partial_s G (d,s) \equiv 0$. The other relevant partial derivatives are
	\begin{align}
		\partial_s G(d,s) & = \frac1{1-s} \sum_{j=0}^{n-1} (j+1) s^j p_j (\frac{4s}{(1+s)^2})+ s q(s)\\
		\partial_s^2 G(d,s) & = 2 \sum_{j=0}^{n-1} (j+1) s^j \frac{\dd}{\dd s} \left[ \frac1{1-s} p_j \left(\frac{4s}{(1+s)^2}\right)\right] \\
		& \hspace{3cm} + \sum_{j=1}^{n-1} (j+1) j s^{j-1} \frac1{1-s} p_j \left(\frac{4s}{(1+s)^2}\right) + s r(s)
	\end{align}
	where $q$ and $r$ are rational functions. Using that $p_0 \equiv 1$ and $p_1(x) = 4x - 1$, it is straightforward to evaluate $\partial_s G(0,0) = 1$ and $\partial_s^2 G(0,0) = 0$. As before, by differentiating $G(d,s) = 0$ implicitly, it follows that $s'(0) = 1$ and $s''(0) = 0$. Considering that the unique solution $|\xi|$ of~\eqref{eq:impl_func} for positive $c$ tends to zero as $c \to \infty$, the above considerations lead to the expansion
	\begin{equation}
		|\xi|^2 (c) = \frac1c + \mathcal O \left( \frac1{c^3}\right), \quad c \to \infty.
	\end{equation}
	Finally, the claims for $\la (c)$ follow by inserting the derived expansions in~\eqref{eq:lam_rel_mod_xi} and applying Taylor's theorem.
	
\end{proof}

\subsection*{Acknowledgment}
The authors acknowledge the support of the EXPRO grant No.~20-17749X
of the Czech Science Foundation.

\subsection*{Data availability}
Our manuscript has no associated data.

\subsection*{Conflict of interest}
The authors state that there is no conflict of interest.

\appendix

\renewcommand{\theequation}{A.\arabic{equation}}

\section{Integral identities with Chebyshev polynomials}

To compute two integrals with Chebyshev polynomials needed in our proofs, we use a slight modification of the integral identity ~\cite[3.631 Equ.~8]{Gradshteyn-Ryzhik-2007} which reads
\begin{equation}\label{eq:GR.1}
	\int_0^\frac{\pi}{2} \sin^{\nu-1}(\varphi)
	\cos (2\ell\varphi) \dd \varphi = \frac{(-1)^{\ell}\pi\Gamma(\nu)}{2^{\nu}\Gamma\left(\frac{\nu+1}{2}+\ell\right)\Gamma\left(\frac{\nu+1}{2}-\ell\right)}
\end{equation}
for any $\ell\in\Z$ and $\re\,\nu>0$.

\begin{lemma}\label{lem:cheb_id1}
For all $m,n\in\N$ and $\alpha\in\C$ with $\re\,\alpha>-3/2$, one has
\[
 \int_{-1}^{1} (1-x)^{\alpha}U_{m-1}(x)U_{n-1}(x) \sqrt{1-x^2} \dd x
 =\frac{\pi}{2^{\alpha+1}}(-1)^{m+n}\left[\binom{2\alpha}{\alpha+m-n}-\binom{2\alpha}{\alpha+m+n}\right],
 \]
 where the generalized binomial number is defined in~\eqref{binomial}.
For fixed $m,n \in \N$, the right-hand side of the formula is understood as the respective limit at its removable singularities $\alpha = -1$ and $\alpha = -1/2$.
\end{lemma}

\begin{proof}
Recalling identity~\eqref{eq:U_n_cos}, the substitution $x:=\cos\theta$ yields
\begin{equation}\label{eq:cheb_id1}
\int_{-1}^{1} (1-x)^{\alpha}U_{m-1}(x)U_{n-1}(x) \sqrt{1-x^2} \dd x=\int_0^\pi  (1-\cos \theta)^\alpha \sin(m\theta) \sin (n \theta) \dd \theta .
\end{equation}
Further elementary manipulations result in the formula
\begin{align}
\int_0^\pi  (1-\cos \theta)^\alpha \sin(m\theta) \sin (n \theta) \dd \theta 
&=2^{\alpha}\int_0^\pi \sin^{2 \alpha} \left(\frac \theta 2\right)  \left[ \cos ((m-n)\theta ) - \cos ((m+n) \theta) \right] \frac{\dd\theta}{2} \nonumber\\
&=2^{\alpha}\int_0^{\frac{\pi}{2}} \sin^{2 \alpha}(\varphi)\left[ \cos (2(m-n)\varphi ) - \cos (2(m+n)\varphi) \right] \dd \varphi.\nonumber\\
\label{eq:int_cheb_trig_inproof}
\end{align}
In the last integral, we apply identity~\eqref{eq:GR.1} twice and find that it is equal to
\[
\frac{\pi}{2^{\alpha+1}}(-1)^{m+n}\left[\frac{\Gamma(2\alpha+1)}{\Gamma(\alpha+1+m-n)\Gamma(\alpha+1-m+n)}-\frac{\Gamma(2\alpha+1)}{\Gamma(\alpha+1+m+n)\Gamma(\alpha+1-m-n)}\right].
\]
When rewritten in terms of the binomial numbers, we arrive at the statement. Note that even though formula~\eqref{eq:GR.1} only applies for $\re\,  \alpha > -1/2$, (with $m,n \in \N$ fixed) a straightforward identity argument between holomorphic functions implies the sought equality for $\re \, \alpha > -3/2$. To this end, the left-hand side of~\eqref{eq:cheb_id1} is easily verified to be holomorphic on the half plane $\re \, \alpha > -3/2$ by dominated convergence.
\end{proof}

\begin{lemma}\label{lem:cheb_id2}
For all $n\in\N$ and $\alpha\in\C$ with $\re\,\alpha<3/2$, one has
\[
\int_{-1}^1 \frac{U_{n-1}^2 (x)}{(1-x)^{\alpha}} \sqrt{1-x^2} \dd x=
2^{\alpha-2}\,\frac{\Gamma^{2}(\alpha)}{\Gamma(2\alpha)}\left(1-\frac{(\alpha)_{2n}}{(1-\alpha)_{2n}}\right)\tan(\pi\alpha),
\]
where $(\alpha)_{k}:=\alpha(\alpha+1)\dots(\alpha+k-1)$ is the Pochhhammer symbol. At the removable singularities $\alpha\in\Z/2$, the right-hand side is to be understood as the respective limit. 
\end{lemma}

\begin{proof}
As in~\eqref{eq:def.In}, we denote the integral on the left-hand side of the claimed formula by $I_n(\alpha)$.  Lemma~\ref{lem:cheb_id1} applied with $m=n$ and $\alpha$ replaced by $-\alpha$ leads to
\[
I_{n}(\alpha)
= \frac{\pi \Gamma (1-2 \alpha) }{2^{1-\alpha}}   \left( \frac{1}{\Gamma ^2(1-\alpha)} - \frac{1}{\Gamma (1-\alpha  + 2n)\Gamma(1-\alpha - 2n )}  \right).
\]
Repeatedly using the well known identity $\Gamma(z+1)=z\Gamma(z)$, we arrive at
\[
\frac{1}{\Gamma (1-\alpha  + 2n)\Gamma(1-\alpha - 2n )}=\frac{1}{\Gamma^{2}(1-\alpha)}\frac{(\alpha)_{2n}}{(1-\alpha)_{2n}}
\]
and therefore
\[
 I_{n}(\alpha)=\pi2^{\alpha-1}\frac{\Gamma (1-2 \alpha)}{\Gamma^{2}(1-\alpha)}\left(1-\frac{(\alpha)_{2n}}{(1-\alpha)_{2n}}\right).
\]
Applying the reflection identity~\cite[Equ.~5.5.3]{DLMF}
\[
\Gamma (1-z) \Gamma (z) = \frac{\pi}{\sin(\pi z)}, 
\]
one further gets
\[
\frac{\Gamma (1-2 \alpha)}{\Gamma^{2}(1-\alpha)}=\frac{\sin^{2}(\pi\alpha)}{\pi\sin(2\pi\alpha)}\frac{\Gamma^{2}(\alpha)}{\Gamma(2\alpha)}=\frac{\tan(\pi\alpha)}{2\pi}\frac{\Gamma^{2}(\alpha)}{\Gamma(2\alpha)}.
\]
The claimed formula now readily follows. 

Notice that, for $n \in \N$ fixed, the left-hand side of the claimed identity is analytic for  $\re \, \alpha < 3/2$, while the right-hand side has removable singularities at $\alpha\in-\N_{0}/2$, as well as at $\alpha=1/2$ and $\alpha=1$. The respective formulas at the two positive parameters can be determined as
\[
 \lim_{\alpha\to1/2}I_{n}\left(\alpha\right)=\sqrt{2}\sum_{j=1}^{2n}\frac{1}{2j-1} \quad\mbox{ and }\quad \lim_{\alpha\to1}I_{n}(\alpha)=\pi n. \qedhere
\]
\end{proof}

To compute the Green kernel of the bilaplacian, we use a slight extension of the formula~\cite[3.613 Equ.~$2^6$]{Gradshteyn-Ryzhik-2007}, namely 
\begin{equation}\label{eq:GR.2}
	\int_0^\pi \frac{\cos(l \varphi)}{1-2k \cos \varphi + k^2} \dd \varphi = \frac{\pi k^{l}}{1-k^2}
\end{equation}
for $l \in \N_0$ and $k \in \mathbb D$ (extended by analyticity from the original statement for $k \in [0,1)$).

\begin{lemma}\label{lem:bilap}
	For all $m, n \in \N$ and $\mu \in \C \setminus [-4,4]$ one has
	\begin{equation}
		\int_{-1}^1 \frac{U_{m-1} (x) U_{n-1} (x)}{4 (1-x)^2 - \mu^2} \sqrt {1-x^2} \dd x = \frac\pi2 \frac{\xi\eta}{(1-\xi\eta)(\xi-\eta)} \left( \frac{\xi^{m+n} - \xi^{|m-n|}}{\xi-\xi^{-1}} - \frac{\eta^{m+n} - \eta^{|m-n|}}{\eta-\eta^{-1}} \right),
	\end{equation}
 where $\xi, \eta \in \mathbb D\setminus\{0\}$ are (unique) such that $2 + \mu = \xi + \xi^{-1}$ and $2 - \mu = \eta + \eta^{-1}$.
\end{lemma}

\begin{proof}
	We start by the simple identity
	\begin{equation}
		\frac{1}{4(1-x)^2-\mu^2} = \frac{1}{2\mu} \left( \frac{1}{2(1-x) - \mu} - \frac{1}{2(1-x) + \mu} \right),
	\end{equation}
	which together with the substitution $x := \cos \theta$ and relation~\eqref{eq:U_n_cos} results in 
	\begin{align}
		& \int_{-1}^1 \frac{U_{m-1} (x) U_{n-1} (x)}{4 (1-x)^2 - \mu^2} \sqrt {1-x^2} \dd x \\
		& \hspace{2cm} = \frac1{2\mu} \int_0^\pi \left( \frac{1}{2(1-\cos \theta) - \mu} - \frac{1}{2(1-\cos \theta) + \mu} \right) \sin (m\theta) \sin (n \theta) \dd \theta \\
		& \hspace{2cm} = \frac1{4\mu} \int_0^\pi \left( \frac{\eta}{1-2 \eta \cos \theta + \eta^2 } - \frac{\xi}{1 - 2 \xi \cos \theta + \xi^2} \right) \left[\cos (|m-n|\theta) - \cos ((m+n) \theta)\right] \dd \theta.
	\end{align}
	The claim follows from suitably applying~\eqref{eq:GR.2} to each summand above and observing
	\begin{equation*}
		2 \mu = \xi + \xi^{-1} - (\eta + \eta^{-1}) = (\xi - \eta) (1 - \xi^{-1} \eta^{-1}).  \qedhere
	\end{equation*}
\end{proof}


\begin{thebibliography}{10}

\bibitem{Beckner-1995}
{\sc Beckner, W.}
\newblock Pitt's inequality and the uncertainty principle.
\newblock {\em Proc. Amer. Math. Soc. 123}, 6 (1995), 1897--1905.

\bibitem{Birman-1961}
{\sc Birman, M.~{\v{S}}.}
\newblock On the spectrum of singular boundary-value problems.
\newblock {\em Mat. Sb. (N.S.) 55 (97)\/} (1961), 125--174.

\bibitem{Bogdan-Dyda-2011}
{\sc Bogdan, K., and Dyda, B.~o.}
\newblock The best constant in a fractional {H}ardy inequality.
\newblock {\em Math. Nachr. 284}, 5-6 (2011), 629--638.

\bibitem{Ciaurri-Roncal-2018}
{\sc Ciaurri, {\'{O}}., and Roncal, L.}
\newblock Hardy's inequality for the fractional powers of a discrete
  {L}aplacian.
\newblock {\em J. Anal. 26}, 2 (2018), 211--225.

\bibitem{Devyver-Fraas-Pinchover-2014}
{\sc Devyver, B., Fraas, M., and Pinchover, Y.}
\newblock Optimal {H}ardy weight for second-order elliptic operator: an answer
  to a problem of {A}gmon.
\newblock {\em J. Funct. Anal. 266}, 7 (2014), 4422--4489.

\bibitem{Devyver-Pinchover-2016}
{\sc Devyver, B., and Pinchover, Y.}
\newblock Optimal {$L^p$} {H}ardy-type inequalities.
\newblock {\em Ann. Inst. H. Poincar\'{e} C Anal. Non Lin\'{e}aire 33}, 1
  (2016), 93--118.

\bibitem{DLMF}
Nist digital library of mathematical functions.
\newblock http://dlmf.nist.gov/, Release 1.1.8 of 2022-12-15.
\newblock F.~W.~J. Olver, A.~B. {Olde Daalhuis}, D.~W. Lozier, B.~I. Schneider,
  R.~F. Boisvert, C.~W. Clark, B.~R. Miller, B.~V. Saunders, H.~S. Cohl, and
  M.~A. McClain, eds.

\bibitem{Erdelyi_vol2}
{\sc Erd\'{e}lyi, A., Magnus, W., Oberhettinger, F., and Tricomi, F.~G.}
\newblock {\em Higher transcendental functions. {V}ol. {II}}.
\newblock Robert E. Krieger Publishing Co., Inc., Melbourne, Fla., 1981.
\newblock Based on notes left by Harry Bateman, Reprint of the 1953 original.

\bibitem{Frank-2018}
{\sc Frank, R.~L.}
\newblock Eigenvalue bounds for the fractional {L}aplacian: a review.
\newblock In {\em Recent developments in nonlocal theory}. De Gruyter, Berlin,
  2018, pp.~210--235.

\bibitem{Fritzsche-Grauert-2002}
{\sc Fritzsche, K., and Grauert, H.}
\newblock {\em From holomorphic functions to complex manifolds}, vol.~213 of
  {\em Graduate Texts in Mathematics}.
\newblock Springer-Verlag, New York, 2002.

\bibitem{Gerhat-Krejcirik-Stampach-2023}
{\sc Gerhat, B., Krej\v{c}i\v{r}\'{i}k, D., and \v{S}tampach, F.}
\newblock An improved discrete {R}ellich inequality on the half-line.
\newblock {\em Israel J. Math.\/} (2023).
\newblock to appear.

\bibitem{Gesztesy-Littlejohn-Michael-Wellman-2018}
{\sc Gesztesy, F., Littlejohn, L.~L., Michael, I., and Wellman, R.}
\newblock On {B}irman's sequence of {H}ardy-{R}ellich-type inequalities.
\newblock {\em J. Differential Equations 264}, 4 (2018), 2761--2801.

\bibitem{Glazman-1965}
{\sc Glazman, I.~M.}
\newblock {\em Direct methods of qualitative spectral analysis of singular
  differential operators}.
\newblock Israel Program for Scientific Translations, Jerusalem, 1965; Daniel
  Davey \& Co., Inc., New York, 1966.
\newblock Translated from the Russian by the IPST staff.

\bibitem{Gradshteyn-Ryzhik-2007}
{\sc Gradshteyn, I.~S., and Ryzhik, I.~M.}
\newblock {\em Table of integrals, series, and products}, seventh~ed.
\newblock Elsevier/Academic Press, Amsterdam, 2007.
\newblock Translated from the Russian, Translation edited and with a preface by
  Alan Jeffrey and Daniel Zwillinger, With one CD-ROM (Windows, Macintosh and
  UNIX).

\bibitem{Gupta-2021-arxiv}
{\sc Gupta, S.}
\newblock One-dimensional discrete {H}ardy and {R}ellich inequalities on
  integers, 2021.
\newblock arXiv:2112.10923 [math.FA].

\bibitem{Gupta-2023-PhD}
{\sc Gupta, S.}
\newblock {\em Discrete functional inequalities on lattice graphs}.
\newblock June 2023.
\newblock PhD thesis.

\bibitem{Gupta-2023}
{\sc Gupta, S.}
\newblock Hardy and {R}ellich inequality on lattices.
\newblock {\em Calc. Var. Partial Differential Equations 62}, 3 (2023), Paper
  No. 81, 18.

\bibitem{Hansmann-Krejcirik-2022}
{\sc Hansmann, M., and Krej{\v c}i\v{r}{\' i}k, D.}
\newblock The abstract {B}irman--{S}chwinger principle and spectral stability.
\newblock {\em J. Anal. Math. 148\/} (Oct. 2020), 361--398.

\bibitem{Herbst-1977}
{\sc Herbst, I.~W.}
\newblock Spectral theory of the operator
  {$(p\sp{2}+m\sp{2})\sp{1/2}-Ze\sp{2}/r$}.
\newblock {\em Comm. Math. Phys. 53}, 3 (1977), 285--294.

\bibitem{Huang-Ye-2022-arxiv}
{\sc Huang, X., and Ye, D.}
\newblock One dimensional sharp discrete {H}ardy-{R}ellich inequalities, 2022.
\newblock arXiv:2212.12680 [math.AP].

\bibitem{Keller-Nietschmann-2023}
{\sc Keller, M., and Nietschmann, M.}
\newblock Optimal {H}ardy inequality for fractional {L}aplacians on the
  integers.
\newblock {\em Ann. Henri Poincar\'e\/} (Apr. 2023).

\bibitem{Keller-Pinchover-Pogorzelski-2018-AMM}
{\sc Keller, M., Pinchover, Y., and Pogorzelski, F.}
\newblock An improved discrete {H}ardy inequality.
\newblock {\em Amer. Math. Monthly 125}, 4 (2018), 347--350.

\bibitem{Keller-Pinchover-Pogorzelski-2018}
{\sc Keller, M., Pinchover, Y., and Pogorzelski, F.}
\newblock Optimal {H}ardy inequalities for {S}chr\"{o}dinger operators on
  graphs.
\newblock {\em Comm. Math. Phys. 358}, 2 (2018), 767--790.

\bibitem{Keller-Pinchover-Pogorzelski-2020}
{\sc Keller, M., Pinchover, Y., and Pogorzelski, F.}
\newblock Criticality theory for {S}chr\"{o}dinger operators on graphs.
\newblock {\em J. Spectr. Theory 10}, 1 (2020), 73--114.

\bibitem{Krejcirik-Stampach-2022}
{\sc Krej\v{c}i\v{r}\'{i}k, D., and \v{S}tampach, F.}
\newblock A sharp form of the discrete {H}ardy inequality and the
  {K}eller-{P}inchover-{P}ogorzelski inequality.
\newblock {\em Amer. Math. Monthly 129}, 3 (2022), 281--283.

\bibitem{Yafaev-1999}
{\sc Yafaev, D.}
\newblock Sharp constants in the {H}ardy-{R}ellich inequalities.
\newblock {\em J. Funct. Anal. 168}, 1 (1999), 121--144.

\end{thebibliography}
\end{document}